\documentclass[12pt]{article}%
\usepackage{amsmath, amsfonts, amsthm, color,latexsym}
\usepackage{amsmath}
\usepackage{amsfonts}
\usepackage{amssymb}
\usepackage[all] {xy}
\usepackage{color} 
\allowdisplaybreaks[4]
\usepackage{graphicx}%
\providecommand{\U}[1]{\protect\rule{.1in}{.1in}}
\newtheorem{teo}{Theorem}[section]
\newtheorem{prop}[teo]{Proposition}
\newtheorem{cor}[teo]{Corollary}
\newtheorem{ex}[teo]{Example}
\newtheorem{obs}[teo]{Remark}

\newtheorem{lema}[teo]{Lemma}
\newtheorem{final remark}[teo]{Final Remark}
\newtheorem{definition}[teo]{Definition}
\newcommand{\an}{\left \Vert} 

\newcommand{\fn}{\right \Vert} 

\newcommand{\ap}{\left (} 

\newcommand{\fp}{\right )} 

\begin{document}

\title{\sc Coherence and compatibility: a stronger approach}
\date{}
\author{Joilson Ribeiro\thanks{Corresponding author - joilsonor@ufba.br}~, Fabricio Santos\thanks{Supported by a CAPES doctoral scholarship} ~and Ewerton Torres\thanks{Supported by a CAPES postdoctoral scholarship.\hfill\newline2010 Mathematics Subject
Classification: 47L22, 46G25, 47L20, 47B10. \newline Keywords: Banach spaces, multilinear operators, multi-ideals, symmetric ideals.}}\maketitle

\begin{abstract}
In this work we present a definition for coherence and compatibility of multilinear mappings and homogenous polynomial classes. These definitions are more restricted than the ones proposed before. We began analyzing this new definition in a technical sense, searching for what it has in common with other approaches. Then, we moved on to a more practical analysis. Through numerous examples of different classes of multilinear mappings and homogenous polynomials we checked the limits on these proposed definitions and propose there is a vast field in which they apply.
\end{abstract}

\section{Introduction and background}

When we think in terms of multilinear mappings and/or homogenous polynomial classes, it makes sense to ask how such classes compare with each other in distinct degrees of linearity (for multilinear mappings) or homogeneity degrees (for homogenous polynomials), a comparison that is made in general through neighboring levels (or degrees). Such a study is called a coherence study of the class. A comparison that also deserves attention is that of an $n$-linear applications (or $n$-homogeneous polynomials) class with an $1$-linear applications class (which is the same for both multilinear applications and homogeneous polynomials) and is known as a class compatibility study.
It is worth noting that it is not always a simple task to associate a multi-ideal to an operator ideal. For example, the ideal of absolutely summing operators has, at least, eight possible extensions to higher degrees (see, for example, \cite{bpr2,cp19, dimant,matos47,matos48,ps54,pg60}). The study of coherence and compatibility was motivated by  the concepts of ideals of polynomials closed under differentiation and closed for scalar multiplication that were introduced in \cite{bp12} (see also \cite{bbjp}) as an attempt to identify a set of properties that polynomial ideals are expected to have in order to maintain some harmony between the different levels of homogeneity. In that spirit, Carando, Dimant and Muro (related papers can be seen in \cite{cdm22,cdm23}) introduced a notion of coherent and compatible polynomial ideals with the same aim of filtering good polynomial extensions of given operator ideals. 
One question that arises at this point is: Should we develop the study of the coherence and compatibility of multilinear mappings and homogenous polynomials independently or should we compare these classes in any sense?
Our approach starts from the principle, defended by Pellegrino and Ribeiro in \cite{joilson}, that yes, we must do a comparative and unified study of these classes. It would theoretically be possible to develop this study in more general classes, but due to the lack of applicability and also due to the good properties of the ideals, we will work with classes of multi-ideal mappings and classes of homogeneous polynomials ideals. Our work is motivated in the interest of comparing distinct levels of linearity that are not neighbors and this comparison is made through forms (scalar polynomials) rather than the product of linear functionals (linear functional powers).
One of the motivations for this approach is the concept of hyper-ideals, introduced in \cite{ewerton}, where multilinear mappings classes satisfy a stronger property than the multi-ideal property. Instead of the class being stable by composition with linear operators, the class is stable by composition with multilinear mappings, and such a concept will be adapted to homogenous polynomials.
This new approach of coherence and compatibility places the forms in a protagonism previously occupied by linear functionals.
We call these new concepts {\it strong coherence} and {\it strong compatibility}.

In Section 1 we formally present this definition and deduce some general properties of this new concept, comparing especially with the approach presented in \cite{joilson}. In the following sections, through a series of examples, we show that well-known classes satisfy this new concept with some kind of particularity. For example, the composition ideals, introduced in \cite{bpr2} and studied in Section 2, are strongly coherent and compatible regardless of taking, in the polynomial case, the adapted concept from multilinear mappings or the ideal of polynomials obtained naturally through one of the conditions of the definition. However, the inequality method, presented in \cite{ewerton2} and studied in Section 5, is strongly coherent and compatible only in the second way. A class that also deserves to be highlighted is the Dunfort-Pettis class, introduced in Section 4, since it is one of the classes that shows the independence of the conditions of definition and does not satisfy the hyper ideal property, which is surprising from an intuitive perspective, as will be seen throughout the text.

From now on, $E, F, G, H, E_n, G_n, n \in \mathbb{N}$, shall denote Banach spaces over $\mathbb{K} = \mathbb{R}$ or $\mathbb{C}$. The symbols $E'$ stands for the topological dual of $E$ and $B_E$ for its closed unit ball. By ${\cal L}(E_1, \ldots, E_n;F)$ we denote the Banach space of continuous $n$-linear operators from $E_1 \times \cdots \times E_n$ to $F$ endowed with the usual uniform norm $\|\cdot\|$. In the linear case, we write ${\cal L}(E;F)$. If $E_1  = \cdots = E_n$, we write ${\cal L}(^nE;F)$. If $F = \mathbb{K}$ we write ${\cal L}(E_1, \ldots, E_n)$ and ${\cal L}(^nE)$. By ${\cal P}(^nE;F)$ we mean the Banach space of all continuous $n$-homogeneous polynomials from $E$ to $F$ endowed with the usual sup norm, which shall be denoted by $\|\cdot\|$. Given $\varphi_1 \in E_1', \ldots, \varphi_n \in E_n'$ and $y \in F$, by $\varphi_1 \otimes \cdots \otimes \varphi_n \otimes y$ we mean the $n$-linear operator defined by
$$\varphi_1 \otimes \cdots \otimes \varphi_n \otimes y(x_1, \ldots, x_n) = \varphi_1(x_1) \cdots \varphi_n(x_n)y. $$
Linear combinations of such operators are called {\it $n$-linear operators of finite type}. A vector-valued map is said to be of {\it finite rank} if its range generates a finite dimensional subspace of the target space.
The same concept goes for homogenous polynomials, that is, given $\varphi \in E'$ and $y \in F$ by $\varphi^n \otimes y$ we mean the $n$-homogenous polynomial
$$\varphi^n \otimes y(x) = \varphi(x)^ny.$$ Linear combinations of such operators are called {\it $n$-homogenous mappings of finite type}. Analogously, the finite rank polynomials are defined. Given $P\in \mathcal{P}(^nE;F)$, for $\check{P}$, we denote the only symmetrical multilinear mapping in $\mathcal{L}(^nE;F)$ such that $$P(x)=\check{P}(x,\ldots,x).$$ For $a \in E$, we also denoted by $P_a$ the $(n-1)$-homogenous polynomial such that $$P_a(x)=\check{P}(a,x,\ldots,x).$$
For background on spaces of multilinear and/or homogenous polynomials we refer to \cite{dineen, mujica}.


Given a class $\cal M$ ($\cal Q$, resp.) of multilinear mappings (homogenous polynomials) between Banach spaces, by ${\cal M}_n$ we mean its $n$-linear ($n$-homogenous) component, that is, for all Banach spaces $E_1,\ldots,E_n,E$ and $F$, ${\cal M}_n(E_1,\ldots,E_n;F) := {\cal L}(E_1,\ldots,E_n;F) \cap {\cal M}$ (${\cal Q}_n(^nE;F) := {\cal P}(^nE;F) \cap {\cal Q}$)

An $p$-normed multi-ideal is a class $\cal M$ of multilinear mappings endowed with a map $\|\cdot\|_{\mathcal{M}} \colon \mathcal{M} \longrightarrow [0,\infty)$ such that:\\
$\bullet$ For all $n, E_1, \ldots, E_n,F$, $(\mathcal{M}(E_1,\ldots, E_n;F), \|\cdot\|_{\mathcal{M}})$ is a $p$-normed linear subspace of $\mathcal{L}(E_1,\ldots, E_n;F)$ containing the $n$-linear mappings of finite type;\\
$\bullet$  $\|I_n \colon \mathbb{K}^n\longrightarrow \mathbb{K}, I_n(\lambda_1,\ldots,\lambda_n)=\lambda_1\cdots\lambda_n\|_{\cal M} =1$ for every $n$;\\
$\bullet$ The multi-ideal property: If $A \in \mathcal{M}(E_1,\ldots, E_n;F)$, $u_1\in \mathcal{L}(G_{1};E_1)$, $\ldots$, $u_n\in \mathcal{L}(G_n;E_n)$ and $t \in \mathcal{L}(F;H)$, then  $t\circ A\circ(u_1,\ldots,u_n) \in \mathcal{M}(G_1,\ldots, G_n;H)$ and
$$\|t\circ A\circ(u_1,\ldots,u_n)\|_{\mathcal{M}}\le\|t\|\cdot\|A\|_{\mathcal{M}}\cdot
\|u_1\|\cdots\|u_n\|$$
The notions of normed, Banach and $p$-Banach multi-ideals are defined in the obvious way and goes back to \cite{pietsch}. Now

A $p$-normed polynomial is a class $\cal Q$ of homogenous polynomials endowed with a map $\|\cdot\|_{\mathcal{Q}} \colon \mathcal{Q} \longrightarrow [0,\infty)$ such that:\\
$\bullet$ For all $n$, $E,F$, $(\mathcal{Q}(^nE;F), \|\cdot\|_{\mathcal{Q}})$ is a $p$-normed linear subspace of $\mathcal{Q}(^nE;F)$ containing the $n$-homogenous polynomials of finite type;\\
$\bullet$  $\|\hat{I}_n \colon \mathbb{K}\longrightarrow \mathbb{K}, \hat{I}_n(\lambda)=\lambda^n\|_{\cal Q} =1$ for every $n$;\\
$\bullet$ The ideal property: If $P \in \mathcal{Q}(^nE;F)$, $u\in \mathcal{L}(G;E)$ and $t \in \mathcal{L}(F;H)$, then  $t\circ P\circ u \in \mathcal{Q}(^nG;H)$ and
$$\|t\circ P\circ u\|_{\mathcal{Q}}\le\|t\|\cdot\|P\|_{\mathcal{Q}}\cdot
\|u\|^n$$
The notions of normed, Banach and $p$-Banach polynomial ideals are also defined in the intuitive way.


\section{Definition and properties}

\begin{definition}\label{HCoh}\rm
Let $\mathcal{M}$ be a class of normed multilinear operators and $\mathcal{U}$ a class of normed homogeneous polynomials and $N \in \mathbb{N} \cup \{\infty \}$. The sequence $\left(\mathcal{U}_k, \mathcal{M}_k  \right)_{k=1}^{N}$, where $k$ indicates the linearity degree (homogeneity degree) of the mappings (polynomials) that are in $\mathcal{M}$ ($\mathcal{U}$, respectively) with $\mathcal{U}_1 = \mathcal{M}_1 = \mathcal{I}$, is \textit{strongly coherent} if there are constants $\beta_1, \beta_2$ and $\beta_3$ such that, for all Banach spaces $E$, $E_{1},\dots,E_{k+1}$ and $F$, the following conditions are true for all $k = 1,..., N-1$:
\begin{description}
\item $($CH1$)$ If $T \in \mathcal{M}_{k + 1}\left(E_1\ldots, E_{k+1}; F \right)$ and $a_j \in E_j$ for $j=1,\dots, k+1$, then
$$T_{a_j} \in \mathcal{M}_k\left(E_1,\ldots, E_{j-1}, E_{j+1},\ldots, E_{k+1}; F  \right)$$ and
$$\left\|T_{a_j} \right\|_{\mathcal{M}_k} \le \beta_1 \left\|T  \right\|_{\mathcal{M}_{k + 1}}\|a_j\|.$$

\item $($CH2$)$ If $P \in \mathcal{U}_{k+1}\left(^{k+1}E; F  \right)$, $a \in E$, then $P_a$ belongs to $\mathcal{U}_k\left(^kE; F \right)$ with
$$\left\|P_a  \right\|_{\mathcal{U}_k} \le \beta_2 \max{\left\{\left\|\check{P} \right\|_{\mathcal{M}_{k+1}}, \left\|P  \right\|_{\mathcal{U}_{k+1}}  \right\}}\|a\|.$$

\item $($CH3$)$ If $T \in \mathcal{M}_k(E_1,\ldots, E_k; F), Q \in \mathcal{L}\left(E_{k+1},\ldots, E_{k+n} \right)$, then
$$QT \in \mathcal{M}_{k+n}(E_1,\dots, E_{k+n}; F)$$
 and
$$\left\|Q T\right\|_{ \mathcal{M}_{k + n}} \le \beta_3 \|Q \| \left\|T  \right\|_{ \mathcal{M}_{k}}.$$

\item $($CH4$)$ If $P \in \mathcal{U}_{k}\left(^kE; F  \right)$ and $Q \in \mathcal{P}\left(^nE\right)$, then
$$Q P \in \mathcal{U}_{k+n}\left(^{k + n}E; F  \right).$$

\item $($CH5$)$ For all $k = 1,\ldots, N$, $P$ belongs to $ \mathcal{U}_{k}\left(^kE; F\right)$ if, and only if, $\check{P}$ belongs to $ \mathcal{M}_{k}(^kE; F)$.
\end{description}
\end{definition}

\begin{definition}\label{HComp}\rm
Let $\mathcal{M}$ be a class of normed multilinear operators and $\mathcal{U}$ a class of normed homogenous polynomials and $N \in \mathbb{N} \cup \{\infty \}$. The sequence $\left(\mathcal{U}_k, \mathcal{M}_k  \right)_{k=1}^{N}$,  with $\mathcal{U}_1 = \mathcal{M}_1 = \mathcal{I}$, is \textit{strongly compatible} with $\mathcal{I}$ if there are constants $\alpha_1$, $\alpha_2$ and $\alpha_3$ such that, for all Banach spaces $E$ and $F$, the following conditions are true for all $n \in \{2,\ldots, N\}$:
\begin{description}
\item $($CP1$)$ If $k \in \{1,\ldots, n\}, T \in \mathcal{M}_n(E_1,\ldots, E_n;F)$ and $a_j \in E_j$, for all $j \in \{1,\dots, n\}\backslash{k}$, then
$$T_{a_1,..., a_{k-1},a_{k+1},..., a_n} \in \mathcal{I}(E_k; F)$$ and
$$\|T_{a_1,..., a_{k-1},a_{k+1},..., a_n}\|_{\mathcal{I}} \le \alpha_1 \left\|T\right\|_{\mathcal{M}_n}\|a_1\|\cdots \|a_{k-1}\| \ \|a_{k+1}\|\cdots\|a_n\|.$$

\item $($CP2$)$ If $P \in \mathcal{U}_n(^nE; F)$ and $a \in F$, then $P_{a^{n-1}} \in \mathcal{I}(E;F)$ and
$$\left\|P_{a^{n-1}}\right\|_{\mathcal{I}} \le \alpha_2 \max{\left\{\left\|\check{P} \right\|_{\mathcal{M}_n}, \left\|P  \right\|_{\mathcal{U}_n} \right\}}\|a\|^{n-1},$$
where $P_{a^{n-1}}(x)=\check{P}(a,\dots,a,x)$.

\item $($CP3$)$ If $u \in \mathcal{I}(E_{n}; F)$ and $Q \in \mathcal{L}\left(E_1,\dots,E_{n-1} \right)$, then
$$Qu \in \mathcal{M}_n(E_1,..., E_n; F)\ \ \mbox{and} \ \ \left|\left|Qu  \right|\right|_{\mathcal{M}_n} \le \alpha_3 \|Q\|\left|\left|u  \right|\right|_{\mathcal{I}}.$$

\item $($CP4$)$ If $u \in \mathcal{U}(E; F)$ and $P \in \mathcal{P}\left(^{n-1}E \right)$, then $$Pu \in \mathcal{U}_n(^nE; F).$$

\item $($CP5$)$ $P$ belongs to $\mathcal{U}_n(^nE; F)$ if, and only if, $\check{P}$ belongs to $\mathcal{M}_n(^nE; F)$.
\end{description}
\end{definition}

\begin{obs}\label{rmrk}\rm\
\begin{description}
\item $(1)$ It is clear that the above definitions are more restrictive than \cite[Definition 3.1 and Definition 3.2]{joilson} regarding the conditions (CH3) and (CP3).

\item $(2)$ As previously done in \cite{joilson}, we do not use the norm control conditions in (CH4) and (CP4) because, under these conditions, some canonical pairs of ideals would not be strongly coherent nor strongly compatible.

\item $(3)$ If $\beta_1 = \beta_2 = \beta_3 = 1$, then the strong coherence of a sequence $\left(\mathcal{U}_k, \mathcal{M}_k  \right)_{k=1}^{N}$ implies the strong compatibility with $\mathcal{I}$.

\item $(4)$ Regarding the notation of the degree of linearity, in the multilinear case, or homogeneity, in the polynomial case, we will reserve the option of overwriting that index when we find it appropriate to the overall notation.

\end{description}

\end{obs}

Before we move on to explore this concept through numerous examples we noted that given a normed class $(\mathcal{M};\|\cdot\|_\mathcal{M})$ of multilinear mappings there is a natural choice for the polynomial class that could permit the pair becoming coherent or compatible. It is a well-known class, the definition of which we recall below.

\begin{definition}\rm Let $(\mathcal{M};\|\cdot\|_\mathcal{M})$ be a normed class of multilinear mappings. Then we considered the class
\begin{equation}\label{EE1.1.}
\mathcal{P}_{\mathcal{M}} := \left\{P \in \mathcal{P}; \check{P} \in \mathcal{M} \right\}
\end{equation}
with the norm $$\|P\|_{\mathcal{P}_{\mathcal{M}}}:=\|\check{P}\|_{\mathcal{M}},$$ inherited from the class $\mathcal{M}$.\end{definition}

Before we prove the main result of this section, we need to recall an important notion in the class of multilinear mappings. Let $S_n$ denote the group of permutations of $\{1,\ldots,n\}$. Given a multilinear operator $A\in\mathcal{L}(^nE;F)$ and a $\sigma\in S_n$, define $A_\sigma$ and  $A_s$ in $\mathcal{L}(^nE;F)$ by $$A_\sigma(x_1,\ldots,x_n)=A(x_{\sigma(1)},\ldots,x_{\sigma(n)}),$$
$$A_s(x_1,\ldots,x_n)=\frac{1}{n!}\cdot\sum\limits_{\sigma\in S_n}A(x_{\sigma(1)},\ldots,x_{\sigma(n)})=\frac{1}{n!}\cdot\sum\limits_{\sigma\in S_n}A_\sigma(x_1,\ldots,x_n).$$
\indent We say that a subclass $\mathcal{G}$ of the class of continuous multilinear operators between Banach spaces is {\it symmetric} if $A_s\in\mathcal{G}(^nE;F)$ whenever $A\in\mathcal{G}(^nE;F).$ If $\cal G$ is endowed with a function $\|\cdot\|_\mathcal{G}\colon \mathcal{G}\longrightarrow[0,\infty)$, we say that $\mathcal{G}$ is {\it strongly symmetric} if $A_\sigma \in {\cal G} {\rm ~and~} \|A_\sigma\|_\mathcal{G}=\|A\|_\mathcal{G}$ for all $A\in\mathcal{G}(^nE;F)$ and $\sigma\in S_n$.

\begin{prop}\label{PP1.9.}
Let $\left(\mathcal{M}_n, \|\cdot \|_{\mathcal{M}_n}\right)_{n=1}^N$ be a sequence of symmetrical multi-ideals satisfying the conditions $($CH1$)$ and $($CH3$)$  of Definition \ref{HCoh}. Then, the sequence of homogenous polynomials $\left(\mathcal{P}_{\mathcal{M}_n}, \|\cdot \|_{\mathcal{P}_{\mathcal{M}_n}} \right)_{n=1}^N$ satisfies the conditions $($CH2$)$ and $($CH4$)$.\end{prop}

\begin{proof}We begin with condition (CH2). Let $E, F$ be Banach spaces, $P \in \mathcal{P}_{\mathcal{M}_{n+1}}(^{n+1}E; F)$ and $a \in E$. By (CH1), we have $$(P_a)^{\vee}= \check{P}_a \in \mathcal{M}_n.$$
From the definition of $\mathcal{P}_{\mathcal{M}_n}$, it follows that $$P_a \in \mathcal{P}_{\mathcal{M}_{n}}.$$
Now, again from (CH1),
$$\|P_a\|_{\mathcal{P}_{\mathcal{M}_n}} = \|(P_a)^{\vee}\|_{\mathcal{M}_n}=\|\check{P}_a\|_{\mathcal{M}_n}\le \beta_1\|\check{P}\|_{\mathcal{M}_{n+1}}\|a\|.$$
Thus we only need to take $\beta_2 = \beta_1$.

Now we verify the condition (CH4). Let $P \in \mathcal{P}_{\mathcal{M}_n}(^nE; F)$ and $Q \in \mathcal{P}(^mE)$. Note that,
\begin{align*}
&(QP)^{\vee}(x_1,\dots, x_{n+m}) =\\
&\displaystyle\frac{\check{Q}(x_{n+1},\dots, x_{n+m})\check{P}(x_1,\dots, x_{n})+ \cdots + \check{Q}(x_1,\dots, x_n)\check{P}(x_{n+1},\dots, x_{n+m})}
{
	\left(\begin{array}{c}
	n+m\\
	m
	\end{array}\right)
}=\\
& \displaystyle\frac{\displaystyle\sum_{\sigma \in S_{n+m}}\check{Q}\check{P}(x_{\sigma(1)},\dots, x_{\sigma(n+m)})}
{n!m!\frac{(n+m)!}{n!m!}}= \left( \check{Q}\check{P}\right)_{s}(x_1,\dots, x_{n+m}).
\end{align*}
So, $(QP)^{\vee} = \left( \check{Q}\check{P}\right)_{s}$.  Because $(\mathcal{M}_n)_{n=1}^N$ satisfies $($CH3$)$ and $\check{P}\in\mathcal{M}_n$ (by definition of $\mathcal{P}_{\mathcal{M}_n}$) we have that
$\check{P}\check{Q}\in\mathcal{M}_{n+m}(^{n+m}E;F)$. Then, it follows from the symmetry of $(\mathcal{M}_n)_{n=1}^N$ that $$(QP)^{\vee}=(\check{P}\check{Q})_s\in\mathcal{M}_{n+m}(^{n+m}E;F).$$
%
%
\end{proof}

{

\begin{obs}\label{rmk2.6.}\ \rm
	\begin{description}
		\item $(1)$ If $\left(\mathcal{M}_n, \|\cdot \|_{\mathcal{M}_n}\right)_{n=1}^N$ is a normed class of multilinear mappings satisfying the conditions $($CH1$)$ and $($CH3$)$, with $\beta_1 = \beta_3 = 1$, then by  Proposition \ref{PP1.9.} and item (3) of Remark \ref{rmrk} we have that $\left( \left(\mathcal{P}_{\mathcal{M}_n}, \|\cdot\|_{\mathcal{P}_{\mathcal{M}_n}} \right), \left(\mathcal{M}_n, \|\cdot\|_{\mathcal{M}_n}\right)\right)_{n=1}^N$ is strongly coherent and strongly compatible with $\mathcal{M}_1 = \mathcal{P}_{\mathcal{M}_1} = \mathcal{I}$.
		
		\item $(2)$ We have an analogous result for coherence and compatibility, as was previously done in \cite{joilson}, by adapting the verification of (CH4) for linear functionals instead of forms.
	\end{description}\end{obs}

Now is the appropriate time to talk about the concept of hyper-ideals of multilinear mappings and homogenous polynomials. The definition of hyper-ideals of multilinear mappings was originally given in \cite{ewerton}, whose adaptation to the homogeneous polynomials is immediate, which will be enunciated below.

\begin{definition}\label{dhip}\rm Let $0<p\le1$, $(\mathcal{Q},\|\cdot\|_{\cal Q})$ be a subclass of all homogenous polynomials with a function and $(C_m)_{m=1}^\infty$ be a sequence of real numbers with $C_m\ge1$ for every $m\in\mathbb{N}$ and $C_1=1$. For all $n\in \mathbb{N}$ and Banach spaces $E$ and $F$, assume that:\\
	(i) the component $$\mathcal{Q}(^nE;F):=\mathcal{P}(^nE;F)\cap \mathcal{Q}$$ is a linear subspace of $\mathcal{P}(^nE;F)$ containing the $n$-homogenous polynomials of finite type,\\
	(ii) the restriction of $\|\cdot\|_{\mathcal{Q}}$ to $\mathcal{Q}(^nE;F)$ is a $p$-norm,\\
	(iii) $\|\widehat{I_n}\colon \mathbb{K} \longrightarrow\mathbb{K}~,~ \widehat{I_n}(\lambda)=\lambda^n\|_{\mathcal{Q}}=1$ for every $n$.\\
We say that $(\mathcal{Q},\|\cdot\|_{\cal Q})$ is a \textit{$p$-normed polynomial $(C_m)_{m=1}^\infty$-hyper-ideal} if it satisfies the\\
 {\bf Hyper-ideal property:} For $n,m \in \mathbb{N}$, and Banach spaces $E$, $F$, $G$ and $H$, if $P\in\mathcal{Q}(^nE;F)$, $Q\in\mathcal{P}(^mG;E)$ and $t\in\mathcal{L}(F;H)$, then $t\circ P\circ Q\in \mathcal{Q}(^{mn}G;H)$ and
$$\|t\circ P\circ Q\|_{\mathcal{Q}}\le (C_m)^n\cdot\|t\|\cdot \|P\|_{\mathcal{Q}}\cdot\|Q\|^n.$$

When $C_m=1$ for every $m\in\mathbb{N}$, we simply say that $(\mathcal{Q},\|\cdot\|_{\cal Q})$ is a {\it $p$-normed polynomial hyper-ideal}. If the components $\mathcal{Q}(^nE;F)$ are complete with respect to the topology generated by $\|\cdot\|_{\mathcal{Q}}$, then $(\mathcal{Q},\|\cdot\|_{\cal Q})$ is called a \textit{$p$-Banach polynomial $(C_m)_{m=1}^\infty$-hyper-ideal}. When $p=1$ we say that $(\mathcal{Q},\|\cdot\|_{\cal Q})$ is a \textit{normed (Banach) polynomial $(C_m)_{m=1}^\infty$-hyper-ideal}. 
\end{definition}

Because this definition is an intermediary concept between polynomial ideals and two-sided ideals, introduced in \cite{ewerton3}, many properties found there can be adapted for polynomials hyper-ideal. An notable exception can be made, regarding the class $\mathcal{P}_\mathcal{M}$, as we will see next.

\begin{definition}\rm Let $\mathcal{G}$ be a subclass of the class of all continuous multilinear operators between Banach spaces endowed with a function $\|\cdot\|_\mathcal{G}\colon \mathcal{G}\longrightarrow\mathbb{R}$. We define $\mathcal{P}^{\mathcal{G}}:=\{P\in\mathcal{P} : \ \mbox{there is} \ A\in\mathcal{G} \ \mbox{such that} \ P=\widehat{A}\},$ and
$$\|P\|_{\mathcal{P}^{\mathcal{G}}}=\inf\{\|A\|_{\mathcal{G}} : \ A \in \mathcal{G} \ \mbox{and} \ P=\widehat{A}\}.$$
\end{definition}

This class naturally forms a polynomial hyper-ideal, when $\mathcal{G}$ is a multilinear hyper-ideal, as ensures the next result.

\begin{prop}\label{hiph} If $({\cal H}, \|\cdot\|_{\cal H})$ is a $p$-normed ($p$-Banach) multilinear hyper-ideal, then $(\mathcal{P}^{\mathcal{H}}, \|\cdot\|_{\mathcal{P}^{\mathcal{H}}})$ is a $p$-normed ($p$-Banach) polynomial $\ap\frac{n^n}{n!}\fp_{n=1}^\infty$-hyper-ideal.\end{prop}

\begin{proof}
We will prove only the hyper-ideal property. Given $t\in \mathcal{L}(F;H)$, $Q\in\mathcal{P}^\mathcal{H}(^nE;F)$ and $R\in\mathcal{P}(^mG;E)$, pick $A\in\mathcal{H}(^nE;F)$ such that $\widehat{A}=Q$ and $B\in\mathcal{L}(^mG;E)$ such that $\widehat{B}=R$. Then $t\circ A\circ (B,\ldots,B)\in\mathcal{H}(^{mn}G;H)$. Since $(t\circ A\circ (B,\ldots,B))^\wedge=t\circ Q \circ R$,
we conclude that $t\circ Q \circ R\in\mathcal{P}^\mathcal{H}(^{mn}G;H)$.
Furthermore,
\begin{align*}\|t\circ Q \circ R\|_{\mathcal{P}^\mathcal{H}}&\le\|t\circ A\circ (B,\ldots,B)\|_\mathcal{H}\le\|t\|\cdot\|A\|_\mathcal{H}\cdot\|B\|^n\\&\le \|t\|\cdot\|A\|_\mathcal{H}\cdot \left(\frac{m^m}{m!}\|\widehat{B}\|\right)^n =\ap\frac{m^m}{m!}\fp^n\|t\|\cdot\|A\|_\mathcal{H}\cdot\|R\|^n.\end{align*}
It follows that $\|t\circ Q \circ R\|_{\mathcal{P}^\mathcal{H}}\le \ap\frac{m^m}{m!}\fp^n\|t\|\cdot\|Q\|_\mathcal{H}\cdot\|R\|^n$.
\end{proof}

Let us give a brief reason why, contrary to the case of multilinear/polynomial ideals, it is not to be expected that $({\cal P}_{\cal H},\|\cdot\|_{{\cal P}_{\cal H}})$ is a polynomial hyper-ideal whenever $(\cal H,\|\cdot\|_{\cal H})$ is a multilinear hyper-ideal. Given a multilinear hyper-ideal $\cal H$, $P \in {\cal P}_{\cal H}(^nE;F)$ and $Q \in {\cal P}(^m G;E)$, $P \circ Q$ does not belong to ${\cal P}_{\cal H}$ in general. Indeed, all we know is that $\check P \in {\cal H}(^nE;F)$, and from the hyper-ideal property of $\cal H$ we can only conclude that $\check P \circ (\check Q, \ldots, \check Q)$ belongs to $\cal H$. For $P \circ Q$ to belong to ${\cal P}_{\cal H}$ we should have $(P \circ Q)^\vee$ in $\cal H$, so everything would work if $(P \circ Q)^\vee = \check P \circ (\check Q, \ldots, \check Q)$. There is no hope for this equality to hold, because the multilinear operator $\check P \circ (\check Q, \ldots, \check Q)$ is not symmetric in general.

Thus, we need to impose an extra condition on the multilinear hyper-ideal $(\cal H,\|\cdot\|_{\cal H})$ to guarantee that $({\cal P}_{\cal H}, \|\cdot\|_{{\cal P}_{\cal H}})$ is a polynomial hyper-ideal. As one could guess, the strongly symmetry is necessary. To arrive at this result the following lemma is needed.

\begin{lema}\label{lemasim}\
\begin{description}
\item $(a)$ A subclass $\mathcal{G}$ of the class of continuous multilinear operators between Banach spaces is symmetric if, and only if,  $\mathcal{P}^\mathcal{G}=\mathcal{P}_\mathcal{G}$.

\item $(b)$ If $0 < p \leq 1$ and $(\mathcal{H},\|\cdot\|_{\cal H})$ is a strongly symmetric $p$-normed multilinear hyper-ideal, then
$$\|P\|_{\mathcal{P}^\mathcal{H}}\leq \|P\|_{\mathcal{P}_\mathcal{H}} \leq (n!)^{\frac{1}{p}-1} \|P\|_{\mathcal{P}^\mathcal{H}}$$
for all $n \in \mathbb{N}$ and $P\in {\cal P}_{\cal H}(^nE;F)$.
\end{description}
\end{lema}

\begin{proof} The item (a) is immediately obvious from the definitions. For the item (b), only the second inequality demands a proof. Let $P\in\mathcal{P}_\mathcal{H}(^nE;F)$ and $A\in\mathcal{H}(^nE;F)$ be such that $\widehat{A}=P$. Since $\check{P}=A_s$ and $\mathcal{H}$ is strongly symmetric, we have
\begin{align*}\|P\|_{\mathcal{P}_\mathcal{H}} &= \left(\|\check P\|_{\mathcal{H}}^p\right)^{1/p} =
\left(\|A_s\|_\mathcal{H}^p\right)^{1/p}=\left(\an\dfrac{1}{n!}\cdot \sum\limits_{\sigma\in S_n}A_\sigma\fn_\mathcal{H}^p\right)^{1/p}\\&\le \dfrac{1}{n!}\cdot\left(\sum\limits_{\sigma\in S_n}\|A_\sigma\|_\mathcal{H}^p\right)^{1/p}=\dfrac{1}{n!}\cdot\left(\sum\limits_{\sigma\in S_n}\|A\|_\mathcal{H}^p\right)^{1/p}=\frac{(n!)^{1/p}}{n!}\|A\|_\mathcal{H}.
\end{align*}
Taking the infimum over all such multilinear operators $A$ to get the desired inequality.\end{proof}

Thus we have:

\begin{prop}\label{corhip} If $(\mathcal{H},\|\cdot\|_{\cal H})$ is a strongly symmetric normed (Banach) multilinear hyper-ideal, then $\mathcal{P}^\mathcal{H}=\mathcal{P}_\mathcal{H}$ isometrically. In particular, $(\mathcal{P}_\mathcal{H},\|\cdot\|_{\mathcal{P}_\mathcal{H}})$ is a normed (Banach) polynomial $\ap\frac{n^n}{n!}\fp_{n=1}^\infty$-hyper-ideal. \end{prop}

\begin{proof} By item (a) of \ref{lemasim} $\mathcal{P}^\mathcal{H}=\mathcal{P}_\mathcal{H}$ and by the item (b) of the same lemma, with $p=1$, $\|\cdot\|_{\mathcal{P}_\mathcal{H}}=\|\cdot\|_{\mathcal{P}^\mathcal{H}}$. The result follows from Proposition \ref{hiph}.\end{proof}

With this result we can obtain a version of Proposition \ref{PP1.9.} for hyper-ideals.

\begin{cor}Let $\left(\mathcal{H}_n, \|\cdot \|_{\mathcal{H}_n}\right)_{n=1}^\infty$ be a sequence of symmetrical hyper-ideals satisfying the conditions $($CH1$)$ and $($CH3$)$ of Definition \ref{HCoh}. Then, the sequence of homogenous polynomials $\left(\mathcal{P}_{\mathcal{H}_n}, \|\cdot \|_{\mathcal{P}_{\mathcal{H}_n}} \right)_{n=1}^\infty$ is a hyper-ideal that satisfies the conditions $($CH2$)$ and $($CH4$)$.\end{cor}

Before we finish the section it is worth noting that there are classes with distinct properties that satisfy the Definitions \ref{HCoh} and \ref{HComp}. 
In the examples in the following sections, all classes are at least multi-ideals. For the symmetric classes $(\mathcal{M}_n)_{n=1}^N$ that satisfy the conditions $($CH1$)$ and $($CH3$)$, we can obtain, according to Proposition \ref{PP1.9.}, the conditions $($CH2$)$ and $($CH4$)$ of Definition \ref{HCoh} when we take the class $(\mathcal{P}_{\mathcal{M}_n})_{n=1}^N$, thus concluding that the pair $\left( \left(\mathcal{P}_{\mathcal{M}_n}, \|\cdot\|_{\mathcal{P}_{\mathcal{M}_n}} \right), \left(\mathcal{M}_n, \|\cdot\|_{\mathcal{M}_n}\right)\right)_{n=1}^N$ is strongly coherent. It is also worth noticing that even with condition $($CH5$)$ giving the form of polynomial ideal associated with $\mathcal{M}_n$ in order to become
the pair 
strongly coherent (or compatible), we still have a variety of norms (which are not necessarily equivalent) $\|\cdot\|_n$ that we can choose in $\mathcal{P}_{\mathcal{M}_n}$ that still maintain the pair $\left( \left(\mathcal{P}_{\mathcal{M}_n}, \|\cdot\|_{n} \right), \left(\mathcal{M}_n, \|\cdot\|_{\mathcal{M}_n}\right)\right)_{n=1}^N$ strongly coherent (or compatible), we will see some of these norms in the examples given in the next sections.
Also note that the multi (hyper)-ideal properties are not directly related to the conditions of Definition \ref{HCoh}, \ref{HComp} or even the Definition 3.1 of \cite{joilson}. For instance, we can have multi-ideals that satisfy our definitions of strongly coherent and strongly compatible and hyper-ideals that do not. We will also explore these notions in the followings sections.



\section{Composition ideal}

We begin by remembering the definition of the object that will be studied in this section. Such a definition could for instance be found in \cite{bpr2}.

\begin{definition}\rm
Let $\mathcal{I}$ be an operator ideal and $A\in\mathcal{L}(E_1,\ldots, E_n; F)$. We write $A \in \mathcal{I}\circ \mathcal{L}_{n}(E_1,\ldots, E_n; F)$ if there is a Banach space $G$, a linear operator $u \in \mathcal{I}(G; F)$ and an $n$-linear operator $B \in \mathcal{L}(E_1,\dots, E_n; G)$, such that $$A = u \circ B.$$ If $(\mathcal{I}, \left\|\cdot \right\|_{\mathcal{I}})$ is a $p$-normed ideal, we define
$$\left\|A \right\|_{\mathcal{I} \circ \mathcal{L}} = \inf\{\left\|u \right\|_{\mathcal{I}} \|B\|\},$$ where the infimum is taken above all representations $u \circ B$ of $A$.
\end{definition}

The composition ideal is a hyper-ideal; for more details see \cite{ewerton}. In that work we also see that, when $(\mathcal{I}, \left\|\cdot \right\|_{\mathcal{I}})$ is an $p$-Banach ideal, then the composition ideal is also an $p$-Banach hyper-ideal.

Now, for the polynomial case, we recall the

\begin{definition}\label{CompositionPolynomial}\rm
Given an operator ideal $\mathcal{I}$ and an $n$-homogeneous polynomial $P \in \mathcal{P}(^nE; F)$, we write $P \in \mathcal{I}\circ \mathcal{P}_{n}(^nE; F)$ if there is a Banach space $G$, a linear operator $u \in \mathcal{I}(G; F)$ and a $n$-homogeneous polynomial $B \in \mathcal{P}(^n E; F)$, such that, $$P = u \circ B.$$ If $(\mathcal{I}, \left\|\cdot \right\|_{\mathcal{I}})$ is a $p$-normed ideal, we defined
$$\left\|P \right\|_{\mathcal{I} \circ \mathcal{P}} = \inf{\left\|u \right\|_{\mathcal{I}} \|P\|}.$$
\end{definition}


The next result is immediate. Therefore, the proof will be omitted.

\begin{prop} When $(\mathcal{I}, \left\|\cdot \right\|_{\mathcal{I}})$ is an $p$-Banach ideal, then $\mathcal{I}\circ \mathcal{P}$ will be an $p$-Banach polynomial hyper-ideal.\end{prop}

We already know that $P \in \mathcal{I}\circ \mathcal{P}(^nE; F)$ if, and only if, $\check{P} \in \mathcal{I}\circ \mathcal{L}(E^n; F)$ (see \cite[Proposition 3.2 (b)]{bpr2}) and that the norms $\|\cdot\|_{\mathcal{I}\circ \mathcal{P}}$ and $\|\cdot\|_{\mathcal{P}_{\mathcal{I}\circ \mathcal{L}}}$ are equivalent (see \cite[Proposition 3.7 (b)]{bpr2}). Therefore, we can take the definition of $\mathcal{P}_{\mathcal{I}\circ\mathcal{L}}$ instead of the Definition \ref{CompositionPolynomial}.

We will verify that the sequence  $\left(\mathcal{P}_{\mathcal{I}\circ\mathcal{L}_n}, \mathcal{I}\circ \mathcal{L}_{n} \right)_{n=1}^{N}$ is strongly coherent and strongly compatible with $\mathcal{I}$. This fact will be consequence of the following propositions. The next result can be proved with the same ideas of \cite[Proposition 3.1]{CDM09}.

\begin{prop}Let $T \in \mathcal{I}\circ \mathcal{L}_{n+1}(E_1,\dots, E_{n+1}; F)$ and $a_j \in E_j$, where $j = 1,\dots, n+1$. Then
$$T_{a_j} \in \mathcal{I}\circ \mathcal{L}_{n}(E_1,\dots,E_{j-1}, E_{j+1},\dots, E_{n+1}; F)$$
and
$$\left\|T_{a_j}\right\|_{\mathcal{I}\circ \mathcal{L}_{n}} \le \left\|T\right\|_{\mathcal{I}\circ \mathcal{L}_{n+1}}\|a_j\|_{E_j}.$$
\end{prop}


\begin{prop}
Let $T \in \mathcal{I}\circ \mathcal{L}_{n}(E_1,\dots, E_n; F)$ and $Q \in \mathcal{L}(E_{n+1},\dots, E_{n+m})$, then
$$QT \in \mathcal{I}\circ \mathcal{L}_{n+m}(E_1,\dots, E_{n+m}; F)$$
and $$\left\|QT\right\|_{\mathcal{I}\circ \mathcal{L}_{n+m}} \le \|Q\|\|T\|_{\mathcal{I}\circ \mathcal{L}_{n}}.$$
\end{prop}

\begin{proof}
Let  $T \in \mathcal{I}\circ \mathcal{L}_{n}(E_1,\dots, E_n; F)$ and $Q \in \mathcal{L}(E_{n+1},\dots, E_{n+m})$, there is a Banach space $G$, a linear operator $u \in \mathcal{I}(G; F)$ and an $n$-linear mapping $B : E_1 \times \cdots \times E_n \longrightarrow F$, such that $T=u\circ B$. Therefore, $QT(x_1,\dots, x_{n+m})= u \circ QB (x_1,\dots, x_{n+m}).$
Because $QB \in \mathcal{L}(E_1,\dots, E_{n+m}; F)$, then $QT \in \mathcal{I}\circ \mathcal{L}_{n+m}(E_1,\dots, E_{n+m}; F)$. We also have
$$\left\|QT \right\|_{\mathcal{I}\circ \mathcal{L}_{n+m}} \le \inf{\left\{\|u\|_{\mathcal{I}}.\|QB\|\right\}} = \inf{\left\{\|u\|_{\mathcal{I}}.\|B\|\right\}}\|Q\|= \|T\|_{\mathcal{I}\circ \mathcal{L}_{n}}\|Q\|.$$
where the infimum are taken over every representation of $T=u\circ B$.
\end{proof}

Since $(\mathcal{I}\circ \mathcal{L}_{n})$ is symmetric, it follows from Proposition \ref{PP1.9.} that

\begin{teo}\label{icompcc}
The pair $\left(\mathcal{P}_{\mathcal{I}\circ\mathcal{L}_n}, \mathcal{I}\circ \mathcal{L}_{n} \right)_{n=1}^{N}$ is strongly coherent and compatible with the ideal $\mathcal{I}$.\end{teo}

\begin{ex}\rm Ryan \cite{ryanthesis} proved that $${\cal P}_{\cal K} =  \mathcal{K}\circ\mathcal{P} {\rm ~~ and~~ } {\cal P}_{\cal W} =  \mathcal{W}\circ\mathcal{P}$$ and Pe{\l}czy\'nski \cite{pelczynski} proved that $${\cal L}_{\cal K} =  \mathcal{K}\circ\mathcal{L} {\rm ~~ and~~ } {\cal L}_{\cal W} =  \mathcal{W}\circ\mathcal{L}$$
So, Theorem \ref{icompcc} guarantees that the pairs $\left(\mathcal{P}_{\mathcal{K},n}, \mathcal{L}_{\mathcal{K},n} \right)_{n=1}^{N}$ of compact homogeneous polynomials and multilinear mappings and $\left(\mathcal{P}_{\mathcal{W},n}, \mathcal{L}_{\mathcal{W},n} \right)_{n=1}^{N}$ of weakly compact homogenous polynomials and multilinear mappings are strongly coherent and compatible with $\mathcal{K}$ and $\mathcal{W}$, respectively. The same is true for the pair of finite rank polynomials/multilinear mappings (see \cite[Proposition 3.1(b)]{mujicatrans}) and the approximate by finite rank polynomials/multilinear mappings (see \cite[Theorem 2.2]{leticia}).\end{ex}

\section{Dunfort-Pettis}


In this section we discuss the class of Dunford-Pettis multilinear mappings (for the linear case, we refer to \cite{A79,AB82,B80}), which will allow us to verify the independence of conditions $($CH1$)$ and $($CH3$)$. It is easy to find a multi-ideal that satisfies $($CH1$)$ but not $($CH3$)$; for example, the finite type mappings (or the mappings that can be approximate by finite type, if we want a complete class). 
The study of this class is justified because it has the property to satisfy (CH3) but not (CH1). It follows directly from this that it is not approximable by operators of finite type.

\begin{definition}\rm
Let $E_1, \dots, E_n$ and $F$ be Banach spaces. A continuous multilinear mapping $T : E_1 \times \cdots \times E_n \rightarrow F$ is called \textit{Dunford-Pettis} when $\left(T\left(x_j^{(1)},\dots, x_j^{(n)} \right)\right)_{j=1}^{\infty}$ converges to zero, for any sequences $\left(x_j^{(i)}\right)_{j=1}^{\infty}\subset E_i$, $i=1,\dots, n$, that converges weakly to zero, that is, $$x_j^{(i)} \overset{w}{\longrightarrow} 0 \Rightarrow \left\|T\left(x_j^{(1)},\dots, x_j^{(n)} \right) \right\| \longrightarrow 0.$$
In that case we write $T\in\mathcal{L}_{DP}(E_1,\ldots,E_n;F)$. \end{definition}

\begin{definition}\rm
Let $E$ and $F$ be Banach spaces. A continuous $n$-homogeneous polynomial $P : E \rightarrow F$ is called \textit{Dunford-Pettis} when
$$x_j \overset{w}{\longrightarrow} 0 \Rightarrow \left\|P\left(x_j \right) \right\| \longrightarrow 0,$$ for any sequences $(x_j)_{j=1}^{\infty}\subset E$ that converges weakly to zero. In that case we write $P\in\mathcal{P}_{DP}(^nE;F)$.\end{definition}

An important fact about this class is the following:

\begin{prop}\
	\begin{description}
		\item $(i)$ The class $(\mathcal{L}_{DP};\|\cdot\|)$ is a Banach multi-ideal.
		
		\item $(ii)$ The class $(\mathcal{P}_{DP};\|\cdot\|)$ is a Banach polynomial ideal.
	
	\end{description}
\end{prop}
\begin{proof} We will do the first assertion, the second is analogous. We begin by showing that the finite type mappings are contained in $\mathcal{L}_{DP}$.
Let $T \in \mathcal{L}_{f}(E_1,\dots, E_n; F)$, that is, there is $\varphi_k^{i} \in E_i'$ and $y_k \in F$, such that
$$T(x_1,\dots, x_n) = \sum_{k=1}^m\varphi_k^{1}(x_1)\cdots, \varphi_k^{n}(x_n)y_k.$$
Let $\left(x_j^{(i)}\right)_{j=1}^{\infty} \subset E_i$ weakly convergent to zero. Then
$$T\left(x_j^{(1)}, \dots, x_j^{(n)} \right) = \sum_{k=1}^m \varphi_k^{1}\left(x_j^{(1)}\right)\cdots, \varphi_k^{n}\left(x_j^{(n)}\right)y_k$$
clearly converges to zero.

Now we can check the ideal propriety. Let $G_1,\dots, G_n, E_1,\dots, E_n, F, H $ be Banach spaces, $u_i \in \mathcal{L}(G_i; E_i), i=1,\dots, n, T \in \mathcal{L}_{DP}(E_1,\dots, E_n; F)$ and $t \in \mathcal{L}(F; H)$. Also let $\left(x_j^{(i)} \right)_{j=1}^{\infty} \subset G_i$ be weakly convergent to zero. 
Because $u_i$ is continuous we have that $\left(u_i\left(x_j^{(i)} \right) \right)_{j=1}^{\infty}$ is a weakly convergent sequence to zero in $E_i$, $i=1,\ldots,n$. Since
\begin{align*}
\left\|t \circ T\circ (u_1,\dots, u_n)\left(x_j^{(1)},\dots, x_j^{(n)} \right) \right\| &
&\le \|t\| \left\|T\left(u_1\left(x_j^{(1)} \right),\dots, u_n\left(x_j^{(n)} \right) \right) \right\|
\end{align*}
then, $t \circ T\circ (u_1,\dots, u_n) \in \mathcal{L}_{DP}(G_1,\dots, G_n; H)$.

Now we need to prove that $\mathcal{L}_{DP}(E_1,\dots, E_n; F)$ is closed in $\mathcal{L}(E_1,\dots, E_n; F)$. In order to do that, let
$\left(T_k \right)_{k=1}^{\infty} \subset \mathcal{L}_{DP}(E_1,\dots, E_n; F)$ a convergent sequence, say to $T\in\mathcal{L}(E_1,\dots, E_n; F)$. We have to show that $T$ is Dunford-Pettis. Given $\epsilon > 0$, there is $k_0 \in \mathbb{N}$, such that
\begin{equation*}
k \ge k_0 \Rightarrow \|T_k - T\| < \epsilon.
\end{equation*}
Let $\left(x_j^{(i)} \right)_{j=1}^{\infty} \subset E_i$, $i=1,\dots, n$, weakly convergent to zero sequences. Since all weakly convergent sequences are bounded, there is $ M > 0$, such that $\left\|x_j^{(1)} \right\|\cdots \left\|x_j^{(n)} \right\|\le M$, for every $j \in \mathbb{N}$. Then,
\begin{align*}
\left\|T\left(x_j^{(1)},\dots, x_j^{(n)} \right) \right\| &\le \left\|(T_{k_0} - T)\left(x_j^{(1)},\dots, x_j^{(n)} \right) \right\| + \left\|T_{k_0}\left(x_j^{(1)},\dots, x_j^{(n)} \right) \right\|\\
&\le \|T_{k_0} - T\| \left\|x_j^{(1)} \right\|\cdots \left\|x_j^{(n)} \right\| + \left\|T_{k_0}\left(x_j^{(1)},\dots, x_j^{(n)} \right) \right\|\\
&< \epsilon M + \left\|T_{k_0}\left(x_j^{(1)},\dots, x_j^{(n)} \right) \right\|,
\end{align*}
Thus, $T \in \mathcal{L}_{DP}(E_1,\dots, E_n; F)$.\end{proof}

\begin{obs}
The class $\left(\mathcal{L}_{DP}, \|\cdot \| \right)$ is not a hyper-ideal. Indeed, just consider $T : \ell_2 \times \ell_2 \rightarrow \ell_1$ given by
\begin{equation*}
T\left((x_j)_{j=1}^{\infty}, (y_j)_{j=1}^{\infty} \right) = (x_j y_j)_{j=1}^{\infty}.
\end{equation*}
It is not difficult to see that $T \notin \mathcal{L}_{DP}(\ell_2, \ell_2; \ell_1)$. By Theorem of Schur \cite[Theorem $1.7$]{diestel}. $I_d \in \mathcal{CC}(\ell_1) = \mathcal{L}_{DP}(\ell_1)$, where $\mathcal{CC}$ is the class of the operators completely continuous. So, if $\mathcal{L}_{DP}$ were a hyper-ideal, we would have
\begin{equation*}
T = Id_{\ell_1} \circ T \in \mathcal{L}_{DP}(\ell_2, \ell_2; \ell_1).
\end{equation*}
that is a contradiction.
\end{obs}

\begin{prop} The class $\mathcal{L}_{DP}$ satisfies $($CH3$)$ but not $($CH1$)$.\end{prop}

\begin{proof} Let $T\in\mathcal{L}_{DP}(E_1,\ldots,E_n;F)$ and $Q\in\mathcal{L}(E_{n+1},\ldots,E_{n+m})$. Let also $\left(x_j^{(i)} \right)_{j=1}^{\infty} \subset E_i$, $i=1, \dots, n+m$ be weakly convergent sequences to zero. Thus the sequence $(Q(x_j^{n+1},\ldots,x_j^{n+m}))_{j=1}^\infty$ is bounded and the sequence $(T(x_j^{1},\ldots,x_j^n))_{j=1}^\infty$ converges, in norm, to zero. Then $$(QT(x_j^{1},\ldots,x_j^{n+m}))_{j=1}^\infty$$ converges to zero. Therefore $QT \in \mathcal{L}_{DP}(E_1,\dots, E_{n+m}; F)$.

Let's show that $\mathcal{L}_{DP}$ does not satisfy $($CH1$)$. Consider $E$ a Banach space and take $a \in E\setminus \{0\}$. Then, by the Hahn-Banach Theorem there is $\varphi \in E'$, such that $\vert\varphi(a)\vert = 1$. Define
$T : E \times \ell_2 \rightarrow \ell_2$ by $$T\left(x, (x_j)_{j=1}^{\infty} \right) := (\varphi(x)x_j)_{j=1}^{\infty}.$$ It is easy to see that $T$ is continuous bilinear application. Now if
$x_j \overset{w}{\longrightarrow} 0$ in $E$ and $y_j \overset{w}{\longrightarrow} 0$ in $\ell_2$ then  $\varphi(x_j) \overset{\vert\cdot\vert}{\longrightarrow} 0$ and $(y_j)_{j=1}^\infty$ is bounded in $\ell_2$. Thus $$T(x_j,y_j)\overset{\|\cdot\|_2}{\longrightarrow} 0,$$ that is, $T \in \mathcal{L}_{DP}(E, \ell_2; \ell_2)$.

Now note that $T_a \notin \mathcal{L}_{DP}(\ell_2; \ell_2)$. Indeed, consider $e_j \in \ell_2$, $j \in \mathbb{N}$. It is well known that $e_j$ converges weakly to zero in $\ell_2$. But $$\|T(a, e_j)\|_2=|\varphi(a)|,$$ for all $j \in \mathbb{N}$. Thus $(T_a(e_j))_{j=1}^\infty$ does not converges to zero in $\ell_2$, that is, $\mathcal{L}_{DP}$ does not satisfy (CH1).


\end{proof}


\section{Inequality method}

The inequality method was introduced in \cite{ewerton2} as a class of multilinear mappings that form a hyper-ideal. We recall those definitions and results.

\begin{definition}\rm
	Let $0 <p\leq 1$. By $\mathcal{BAN}$ we denote the class of all Banach spaces over $\mathbb{K}=\mathbb{R}$ or $\mathbb{C}$ and by $p-\mathcal{BAN}$ the class of all $p$-Banach spaces over $\mathbb{K}$. A correspondence $$\mathcal{X}:\mathcal{BAN}\longrightarrow p-\mathcal{BAN}$$ that associates to each Banach space $E$ an $p$-Banach space $\left(\mathcal{X}(E), \Vert\cdot\Vert_{\mathcal{X}(E)}\right)$ is called an $p$-sequence functor if:
	\begin{description}
		\item {(i)} $\mathcal{X}(E)$ is a linear subspace of $\mathbb{E}^{\mathbb{N}}$ with the usual algebraic operations;
		\item {(ii)} For all $x\in E$ and $j\in\mathbb{N}$, we have $(0,\dots,0,x,0,\dots)\in\mathcal{X}(E)$, where $x$ is placed at the $j$th coordinate, and $\Vert(0,\dots,0,x,0,\dots)\Vert_{\mathcal{X}(E)}=\Vert x\Vert_{E}$.
		\item {(iii)} For every $u\in\mathcal{L}(E;F)$ and every finite $E$-valued sequence $(x_j)_{j=1}^{k}:=(x_1,\dots,x_k,0,0,\dots)$, $k\in\mathbb{N}$, it holds $$\left\Vert\left(u(x_j)\right)_{j=1}^{k}\right\Vert_{\mathcal{X}(F)}\leq\Vert u\Vert\left\Vert\left(x_j\right)_{j=1}^{k}\right\Vert_{\mathcal{X}(E)}.$$
	\end{description}
When $p=1$ we simply say that $\mathcal{X}$ is a sequence functor.
\end{definition}

\begin{definition}\rm
Let $0 < p, q \le 1$, $p$-sequence functor $\mathcal{X}$ and $q$-sequence functor $\mathcal{Y}$. We say that an $A \in \mathcal{L}(E_1,\dots, E_n; F)$ is \textit{$(\mathcal{X}-\mathcal{Y})$-summing} if there is a constant $C > 0$ such that
\begin{equation}\label{CI2}
\left\|\left(A(x_j^{(1)},\dots, x_j^{(n)}) \right)_{j=1}^{k} \right\|_{\mathcal{Y}(F)} \le C \sup_{T \in B_{\mathcal{L}(E_1,\dots, E_n)}}\left\|\left(T(x_j^{(1)},\dots, x_j^{(n)}) \right)_{j=1}^k \right\|_{\mathcal{X}(\mathbb{K})},
\end{equation}
for all $k \in \mathbb{N}$ and any finite sequences $\left(x_j^{(i)} \right)_{j=1}^k \subset E_i$, $i=1,\dots, n$. In this case we write $A \in (\mathcal{X}-\mathcal{Y})(E_1,\dots, E_n; F)$.

The infimum of constants $C > 0$ that satisfies the condition \eqref{CI2} is a $q$-norm in $(\mathcal{X}-\mathcal{Y})(E_1,\dots, E_n; F)$, that will noted by $\left\|\cdot \right\|_{(\mathcal{X}-\mathcal{Y})}$.\end{definition}

\begin{definition}\rm Let $0 < p, q \le 1$. We say that an $p$-sequence functor $\mathcal{X}$ is \textit{scalary dominated} by the $q$-sequence functor $\mathcal{Y}$ when, for all finite sequence $\left(\lambda_j \right)_{j=1}^k \subset \mathbb{K}$, $k \in \mathbb{N}$ we have
$$\left\|\left(\lambda_j \right)_{j=1}^k \right\|_{\mathcal{X}(\mathbb{K})} \le \left\|\left(\lambda_j \right)_{j=1}^k \right\|_{\mathcal{Y}(\mathbb{K})}.$$
\end{definition}


\begin{teo}\cite[Theorem 2.7]{ewerton2}
Let $0 < p, q \le 1$, $p$-sequence functor $\mathcal{X}$ and $q$-sequence functor $\mathcal{Y}$, such that $\mathcal{Y}$ is scalary dominated by $\mathcal{X}$. Then, $$\left((\mathcal{X}-\mathcal{Y}), \left\|\cdot \right\|_{(\mathcal{X}-\mathcal{Y})} \right)$$ is an $q$-Banach hyper-ideal.
\end{teo}

The proof of the next result is easily verified, and will be omitted. 

\begin{prop}\label{P1.7.}
Let $T \in (\mathcal{X}-\mathcal{Y})_{n+1}(E_1,\dots, E_{n+1}; F)$ and $a_i \in E_i$, $i=1,\ldots,n$. Then
$$T_{a_i} \in (\mathcal{X}-\mathcal{Y})_{n+1}(E_1,\dots, E_{i-1}, E_{i+1},\dots E_{n+1}; F)$$
and
$$\|T_{a_i}\|_{(\mathcal{X}-\mathcal{Y})_n} \le \|T\|_{(\mathcal{X}-\mathcal{Y})_{n+1}}\|a_i\|.$$\end{prop}

%

\begin{prop}\label{P1.8.}Let $T \in (\mathcal{X}-\mathcal{Y})_{n}(E_1,\dots, E_{n}; F)$ and $Q \in \mathcal{L}(E_{n+1},\dots, E_{n+m})$. Then
$$QT \in (\mathcal{X}-\mathcal{Y})_{n+m}(E_1,\dots, E_{n+m}; F)$$ and $$\|QT\|_{(\mathcal{X}-\mathcal{Y})_{n+m}} \le \|Q\| \ \|T\|_{(\mathcal{X}-\mathcal{Y})_n}.$$\end{prop}

\begin{proof}
Let $T \in (\mathcal{X}-\mathcal{Y})_{n}(E_1,\dots, E_{n}; F)$ and $Q \in \mathcal{L}(E_{n+1},\dots, E_m)$. Then
\begin{align*}
&\left\|\left(QT(x_j^{(1)},\dots, x_j^{(n+m)}) \right)_{j=1}^{k} \right\|_{\mathcal{Y}(F)}\\
&= \left\|\left(T\left(Q\left(x_j^{(n+1)},\dots, x_j^{(n+m)}\right)x_j^{(1)},\dots, x_j^{(n)}\right) \right)_{j=1}^{k} \right\|_{\mathcal{Y}(F)}\\
&\le \|T\|_{(\mathcal{X}-\mathcal{Y})_{n}} \sup_{\varphi \in B_{\mathcal{L}(E_1,\dots, E_n)}}\left\|\left(\varphi \left(Q\left(x_j^{(n+1)},\dots, x_j^{(n+m)}\right)x_j^{(1)},\dots, x_j^{(n)}\right) \right)_{j=1}^k \right\|_{\mathcal{X}(\mathbb{K})}\\
&\le  \|T\|_{(\mathcal{X}-\mathcal{Y})_{n}}\|Q\|\sup_{\psi \in B_{\mathcal{L}(E_1,\dots, E_{n+m})}}\left\|\left(\psi\left(x_j^{(1)},\dots, x_j^{(n+m)}\right) \right)_{j=1}^k \right\|_{\mathcal{X}(\mathbb{K})}.
\end{align*}
Thus, $$QT \in (\mathcal{X}-\mathcal{Y})_{n}(E_1,\dots, E_{n+m}; F)\
\mbox{and} \ \|QT\|_{(\mathcal{X}-\mathcal{Y})_{n+m}} \le \|Q\| \ \|T\|_{(\mathcal{X}-\mathcal{Y})_n}.$$
\end{proof}

We could define the polynomial case in the intuitive way. However, this treatment is not suitable in our context. For instance, in \cite{dimant} it was shown that the polynomial ideal $\left(\mathcal{P}\prod_p^{n, str}, \left\|\cdot \right\|_{\mathcal{P}\prod_p^{n, str}} \right)_{n=1}^N$, generated by this method, is not compatible with the ideal $\prod_p$, according to \cite{joilson}. Thus, we follow the natural approach and work with the class $\mathcal{P}_{(\mathcal{X}-\mathcal{Y})}$. 
This is a hyper-ideal, by Proposition \ref{corhip}, because it is not to difficult to see that the class $(\mathcal{X}-\mathcal{Y})$ is strongly symmetric. Moreover, by Proposition \ref{PP1.9.}, we have that is true the conditions $($CH2$)$ and $($CH4$)$. Therefore, by the Remark \ref{rmrk} item $(3)$ we conclude that:





\begin{teo}
The sequence $\left(\left(\mathcal{P}_{(\mathcal{X}-\mathcal{Y})_n}, \|\cdot\|_{\mathcal{P}_{(\mathcal{X}-\mathcal{Y})_n}} \right), \left((\mathcal{X}-\mathcal{Y})_n, \|\cdot \|_{(\mathcal{X}-\mathcal{Y})_n} \right) \right)_{n=1}^N$ is strongly coherent and compatible with $(\mathcal{X}-\mathcal{Y})_1$.\end{teo}

\section{$\mathcal{I}$-bounded method}

\begin{definition}\rm
Let $\mathcal{I}$ be an operator ideal. We say that a subset $K$ of a Banach space $F$ is \textit{$\mathcal{I}$-bounded} if there is a Banach space $H$ and a linear operator $u \in \mathcal{I}(H; F)$, such that, $K \subset u(B_H)$. The class of $\mathcal{I}$-bounded operators is denoted by $C_{\langle \mathcal{I} \rangle}(F)$.
\end{definition}

\begin{definition}\rm
Let $\mathcal{I}$ be an operator ideal. We say that a mapping $A \in \mathcal{L}(E_1,\dots, E_n; F)$ is \textit{$\mathcal{I}$-bounded} if
$$A(B_{E_1}\times \cdots \times B_{E_m}) \in C_{\langle \mathcal{I}\rangle}(F).$$
In other words, there is a Banach space $H$ and an operator $u \in \mathcal{I}(H; F)$, such that,
\begin{equation}\label{CI}
A(B_{E_1}\times \cdots \times B_{E_m}) \subset u(B_{H}).
\end{equation}
We denoted by $\mathcal{L}_{\langle\mathcal{I}\rangle}$ the class of $\mathcal{I}$-bounded mappings.
If $\|\cdot \|_{\mathcal{I}}$ is an $p$-norm in $\mathcal{I}$, then we can define a norm in $\mathcal{L}_{\langle\mathcal{I}\rangle}$, by
$$\left\|A \right\|_{\mathcal{L}_{\langle\mathcal{I}\rangle}} = \inf\left\{\|u\|_{\mathcal{I}}; u \ \text{satisfies \eqref{CI}} \right\}.$$
\end{definition}

In \cite{ewerton2} was proven the following.

\begin{teo}
Let $0 \le p \le 1$ and $\left(\mathcal{I}, \|\cdot \|_{\mathcal{I}} \right)$  an $p$-Banach ($p$-normed) ideal. Then, $\left(\mathcal{L}_{\langle\mathcal{I}\rangle}, \|\cdot \|_{\mathcal{L}_{\langle\mathcal{I}\rangle}} \right)$ is an $p$-Banach ($p$-normed) hyper-ideal.\end{teo}

Now we are going to show that $\mathcal{L}_{\langle\mathcal{I}\rangle}$ satisfies the conditions of coherence and compatibility introduced at the beginning of the paper.

\begin{prop}
Let $T \in \mathcal{L}_{\langle\mathcal{I}\rangle}^{(n+1)}(E_1,\dots, E_{n+1}; F)$ and $a_j \in E_j$, $j=1,\dots, n+1$. Then
$$T_{a_j} \in \mathcal{L}_{\langle\mathcal{I}\rangle}^{(n)}(E_1,\dots, E_{j-1}, E_{j+1},\dots,E_{n+1};F)$$ and
$$\|T_{a_j}\|_{\mathcal{L}_{\langle\mathcal{I}\rangle}^{(n)}} \le \|T\|_{\mathcal{L}_{\langle\mathcal{I}\rangle}^{(n+1)}}\|a_j\|.$$
\end{prop}

\begin{proof}
As in other proofs on the same subject, we only do the case $j=1$. Let $T \in \mathcal{L}_{\langle\mathcal{I}\rangle}^{(n+1)}(E_1,\dots, E_{n+1}; F)$ and $a_1 \in E_1$, there is a Banach space $H$ and a linear operator $u \in \mathcal{I}(H; F)$, such that, for any $x_j \in E_j, j=2,\dots, n+1$,

$$T_{a_1}(x_2,\dots, x_{n+1}) = T(a_1, x_2,\dots, x_{n+1})= T\left(\frac{a_1}{\|a_1\|}, x_2,\dots, x_{n+1} \right)\|a_1\|.$$

We considered the operator $\|a_1\|u \in \mathcal{I}(H; F)$. In this way,
$$T\left(\frac{a_1}{\|a_1\|}, x_2,\dots, x_{n+1} \right) \in u(B_H) \Rightarrow T\left(\frac{a_1}{\|a_1\|}, x_2,\dots, x_{n+1} \right)\|a_1\| \in \left(\|a_1\|u \right)(B_H).$$
Then,
$$T_{a_1}(x_2,\dots, x_{n+1}) = T\left(\frac{a_1}{\|a_1\|}, x_2,\dots, x_{n+1} \right)\|a_1\| \in \left(\|a_1\|u \right)(B_H).$$
Thus we conclude, $T_{a_1} \in \mathcal{L}_{\langle\mathcal{I}\rangle}^{(n)}(E_2,\dots, E_{n+1}; F)$. Moreover $$\|T_{a_1}\|_{\mathcal{L}_{\mathcal{I}}^{(n)}}\le\an(\|a_1\|u)\fn_\mathcal{I}= \|a_1\|\|u\|_\mathcal{I}.$$ Taking the infimum above all operators $u$ that satisfies \eqref{CI}, we have $$\|T_{a_1}\|_{\mathcal{L}_{\langle\mathcal{I}\rangle}^{(n)}}\le \|T\|_{\mathcal{L}_{\langle\mathcal{I}\rangle}^{(n+1)}}\|a_1\|.$$

\end{proof}

\begin{prop}\label{prop75}
Let $T \in \mathcal{L}_{\langle\mathcal{I}\rangle}^{(n)}(E_1,\dots, E_{n}; F)$ and $Q\in \mathcal{L}(E_{n+1},\dots, E_{n+m})$. Then
$$Q T \in \mathcal{L}_{\langle\mathcal{I}\rangle}^{(n+m)}(E_1,\dots, E_{n+m}; F)$$ and $$\|Q T\|_{\mathcal{L}_{\langle\mathcal{I}\rangle}^{(n+m)}} \le \|Q\|\|T\|_{\mathcal{L}_{\langle\mathcal{I}\rangle}^{(n)}}.$$
\end{prop}

\begin{proof}
Let $T \in \mathcal{L}_{\langle\mathcal{I}\rangle}^{(n)}(E_1,\dots, E_{n}; F)$ and $Q\in \mathcal{L}(E_{n+1},\dots, E_{n+m})$, there is a Banach space $H$ and a linear operator $u \in \mathcal{I}(H; F)$, such that, for any $(x_1,\dots, x_n) \in B_{E_1}\times \cdots \times B_{E_n}$
$$T(x_1,\dots, x_n) \in u(B_H).$$
That is, it exist $h \in B_H$, such that, $T(x_1,\dots, x_n) = u(h)$. Now
\begin{align*}
Q T(x_1,\dots, x_{n+m}) &= Q(x_{n+1},\dots, x_{n+m}) T(x_1,\dots, x_n)\\
&= Q(x_{n+1},\dots, x_{n+m})u(h)\\
&= u\left(Q(x_{n+1},\dots, x_{n+m})h \right)\\
&= u\left(\frac{Q}{\|Q\|}(x_{n+1},\dots, x_{n+m})h \right)\|Q\|,
\end{align*}
where $(x_1,\dots, x_{n+m}) \in B_{E_1}\times \cdots \times B_{E_{n+m}}$. Considering the operator
$$\tilde{u} := \|Q\|u : H \rightarrow F,$$
we have $\tilde{u}\in\mathcal{I}(H;F)$ and
$Q T(x_1,\dots, x_{n+m})=\tilde{u}(\tilde{h})$, where $\tilde{h}=\frac{Q}{\|Q\|}(x_{n+1},\dots, x_{n+m})h \in B_H$. Thus
$$Q T \in \mathcal{L}_{\langle\mathcal{I}\rangle}^{(n+m)}(E_1,\dots, E_{n+m}; F).$$
We also have,

$$\|QT\|_{\mathcal{L}_{\langle\mathcal{I}\rangle}^{(n+m)}}\le \an\tilde{u}\fn_\mathcal{I}= \|Q\|\|u\|_\mathcal{I}.$$ Then, taking the infimum above all operators $u$ that satisfies \eqref{CI}, $$\|QT\|_{\mathcal{L}_{\langle\mathcal{I}\rangle}^{(n+m)}}\le \|T\|_{\mathcal{L}_{\langle\mathcal{I}\rangle}^{(n)}}\|Q\|.$$

\end{proof}

Now we analyze the polynomial case. Using that $A_\sigma(B_E\times\cdots\times B_E)=A(B_E\times\cdots\times B_E)$ it is not difficult to see that $\left(\mathcal{L}_\mathcal{\langle\mathcal{I}\rangle}, \|\cdot \|_{\mathcal{L}_\mathcal{\langle\mathcal{I}\rangle}} \right)$ is a strongly symmetric hyper-ideal. Thus the pair $$\left(\left(\mathcal{P}_{\mathcal{L}_\mathcal{\langle\mathcal{I}\rangle}^{(n)}}, \|\cdot\|_{\mathcal{P}_{\mathcal{L}_\mathcal{\langle\mathcal{I}\rangle}^{(n)}}} \right), \left(\mathcal{L}_\mathcal{\langle\mathcal{I}\rangle}^{(n)} , \|\cdot \|_{\mathcal{L}_\mathcal{\langle\mathcal{I}\rangle}^{(n)}} \right) \right)_{n=1}^{N}$$ is strongly coherent and compatible with $\mathcal{L}_\mathcal{\langle\mathcal{I}\rangle}^{(1)}$, according to Proposition \ref{PP1.9.} and item $(3)$ of \ref{rmrk}. Now we begin the analysis of the case where the definition is independent of the multilinear case.

\begin{definition}\rm
Let $P \in \mathcal{P}(^nE; F)$ and $\mathcal{I}$ be an operator ideal. We say that $P$ is \textit{$\mathcal{I}$-bounded} if $P(B_E) \in C_{\mathcal{I}}$, that is, there is a Banach space $H$ and a linear operator $u \in \mathcal{I}(H; F)$, such that
\begin{equation}\label{E1.14}
P(B_E) \subset u(B_H).
\end{equation}

We denoted the $\mathcal{I}$-bounded polynomial space by $\mathcal{P}_{\mathcal{\langle\mathcal{I}\rangle}}^{(n)}(^nE; F)$. We can also define, in the same fashion as the multilinear mappings, a norm in $\mathcal{P}_{\mathcal{\langle\mathcal{I}\rangle}}^{(n)}(^nE; F)$ through the norm of $\mathcal{I}$ in the following manner. We considered the mapping
$\|\cdot \|_{\mathcal{P}_{\mathcal{\langle\mathcal{I}\rangle}}^{(n)}} : \mathcal{P}_{\mathcal{\langle\mathcal{I}\rangle}}^{(n)} \rightarrow [0, +\infty)$, given by

$$\|P\|_{\mathcal{P}_{\mathcal{\langle\mathcal{I}\rangle}}^{(n)}} = \inf\{u; u \ \text{satisfies \eqref{E1.14}}\}.$$

We denoted the class of $\mathcal{I}$-bounded homogeneous polynomials by $\mathcal{P}_{\mathcal{\langle\mathcal{I}\rangle}}$.
\end{definition}

Adapting the proof of \cite[Theorem 3.3]{ewerton2} for polynomials we get:

\begin{teo}Let $(\mathcal{I},\|\cdot\|_{\cal I})$ be an $p$-Banach operator ideal. Then $(\mathcal{P}_{\langle\mathcal{I}\rangle},\|\cdot\|_{\mathcal{P}_{\langle\mathcal{I}\rangle}})$ is an $p$-Banach polynomial hyper-ideal.\end{teo}

\begin{prop}\label{CH5IL} Let $E$ and $F$ be Banach spaces. Then $P \in \mathcal{P}_{\mathcal{\langle\mathcal{I}\rangle}}^{(n)}(^nE; F)$ if, and only if, $\check{P} \in \mathcal{L}_{\mathcal{\langle\mathcal{I}\rangle}}^{(n)}(^nE; F)$.\end{prop}

\begin{proof} Given $P\in\mathcal{P}_{\mathcal{\langle\mathcal{I}\rangle}}^{(n)}(^nE;F)$, then there exists a Banach space $H$ and a linear operator $u\in\mathcal{I}(H;F)$ such that $P(B_E)\subseteq u(B_H)$. Let $x_1,\ldots,x_n\in B_E$ and $\varepsilon_1, \ldots, \varepsilon_n=\pm1$. Since $\displaystyle\frac{\varepsilon_1x_1 + \cdots + \varepsilon_nx_n}{n} \in B_E$, there is $z_{\varepsilon_1, \ldots, \varepsilon_n} \in B_H$ such that $$P\left(\displaystyle\frac{\varepsilon_1x_1 + \cdots + \varepsilon_nx_n}{n}\right) = u\left(z_{\varepsilon_1, \ldots, \varepsilon_n}\right).$$
Then $w:=\dfrac{1}{2^n}\sum\limits_{\varepsilon_j=\pm1}\varepsilon_1\cdots\varepsilon_nz_{\varepsilon_1,\ldots,\varepsilon_n} \in B_H$. By the Polarization Formula, \begin{eqnarray*}\check{P}(x_1,\ldots,x_n)&=&\dfrac{1}{n!2^n}\sum\limits_{\varepsilon_j=\pm1}\varepsilon_1\cdots\varepsilon_nP(\varepsilon_1x_1+\cdots+\varepsilon_nx_n) \\&=& \dfrac{1}{n!2^n}\sum\limits_{\varepsilon_j=\pm1}\varepsilon_1\cdots\varepsilon_nn^nP\ap\dfrac{\varepsilon_1x_1+\cdots+\varepsilon_nx_n}{n}\fp\\&=&
\dfrac{n^n}{n!2^n}\sum\limits_{\varepsilon_j=\pm1}\varepsilon_1\cdots\varepsilon_nu(z_{\varepsilon_1,\ldots,\varepsilon_n}) =\dfrac{n^n}{n!}u(w) \in \frac{n^n}{n!}u(B_H).\end{eqnarray*}
Since $\frac{n^n}{n!}u \in {\cal I}(H;F)$, we conclude that $\check{P}(B_E\times\cdots\times B_E)\subseteq \frac{n^n}{n!}u(B_H),$ proving that $\check{P} \in \mathcal{L}_{\mathcal{\langle\mathcal{I}\rangle}}^{(n)}(^nE; F)$.
Conversely, the other implication follows easily from the definitions.
\end{proof} 
\begin{prop}
Let $a \in E$ and $P \in \mathcal{P}_{\mathcal{\langle I \rangle}}^{(n+1)}(^{n+1}E; F)$. Then
$$P_a \in \mathcal{P}_{\mathcal{I}}^{(n)}(^{n}E; F)$$ and
$$\|P_a\|_{\mathcal{P}_{\mathcal{\langle I \rangle}}^{(n)}} \le \|\check{P}\|_{\mathcal{L}_{\mathcal{\langle I \rangle}}^{(n+1)}}\|a\|.$$
\end{prop}

\begin{proof}
Let $a \in E$, $a \neq 0$ and $P \in \mathcal{P}_{\mathcal{I}}^{(n+1)}(^{n+1}E; F)$. By Proposition \ref{CH5IL}, $\check{P}\in \mathcal{L}_{\mathcal{\langle\mathcal{I}\rangle}}^{(n+1)}(^{n+1}E; F)$. So, $\check{P}_a \in \mathcal{L}_{\mathcal{\langle\mathcal{I}\rangle}}^{(n)}(^nE; F)$. Since $\check{P}_a = (P_a)^{\vee}$, it follows that $(P_a)^{\vee} \in \mathcal{L}_{\mathcal{\langle\mathcal{I}\rangle}}^{(n)}(^nE; F)$. Therefore, by Proposition \ref{CH5IL}, $P_a \in \mathcal{P}_{\mathcal{I}}^{(n)}(^{n}E; F)$.

Since $\check{P}\in \mathcal{L}_{\mathcal{\langle\mathcal{I}\rangle}}^{(n+1)}(^{n+1}E; F)$, there exists a Banach space $H$ and a linear operator $u \in \mathcal{I}(H; F)$, such that
\begin{equation*}
\check{P}(B_E\times{\cdots}\times B_E) \subset u(B_H)
\end{equation*}
So, $P_a(x) = \check{P}(a, x,{\dots}, x) = \check{P}\left(\frac{a}{\|a\|}, x,\dots, x \right)\|a\|$.
Since $\|a\|u \in \mathcal{I}(H; F)$, we have $P_a(x) \in \|a\|u(B_H)$. So,
\begin{equation*}
\left\|P_a \right\|_{\mathcal{P}_{\mathcal{\langle I \rangle}}} \le \|a\| \|u\|_{\mathcal{I}}
\end{equation*}
for every $u \in \mathcal{I}(H; F)$, such that $\check{P}(B_E\times{\cdots}\times B_E) \subset u(B_H)$. Therefore
\begin{equation*}
\left\|P_a \right\|_{\mathcal{P}_{\mathcal{\langle I \rangle}}} \le \|a\| \left\|\check{P} \right\|_{\mathcal{L}_{\mathcal{\langle I \rangle}}^{(n+1)}}.
\end{equation*}


%
\end{proof}

Using the same argument of Proposition \ref{prop75} we can prove

\begin{prop}
Let $P \in \mathcal{P}_{\mathcal{\langle I \rangle}}^{(n)}(^{n}E; F)$ and $Q \in \mathcal{P}(^{m}E)$. Then
$$PQ \in \mathcal{P}_{\mathcal{\langle I \rangle}}^{(n+m)}(^{n+m}E; F)\ \mbox{and} \
\|PQ\|_{\mathcal{P}_{\mathcal{\langle I \rangle}}^{(n+m)}} \le \|P\|_{\mathcal{P}_{\mathcal{\langle I \rangle}}^{(n)}}\|Q\|.$$
\end{prop}

Thus we can also conclude that the pair $$\left(\left(\mathcal{P}_\mathcal{\langle I \rangle}^{(n)}, \|\cdot\|_{\mathcal{P}_\mathcal{\langle I \rangle}^{(n)}} \right), \left(\mathcal{L}_\mathcal{\langle I \rangle}^{(n)} , \|\cdot \|_{\mathcal{L}_\mathcal{\langle I \rangle}^{(n)}} \right) \right)_{n=1}^{N}$$ is strongly coherent and compatible with $\mathcal{L}_\mathcal{\langle I \rangle}^{(1)}$.


\begin{thebibliography}{99}\small

\vspace*{-0.5em}
\bibitem{A79} K. T. Andrews, {\it Dunford-Pettis sets in the space of Bochner integrable functions}. Mathematische Annalen, {\bf 241} (1979), 35--41.

\vspace*{-0.5em}
\bibitem{AB82} C. D. Arliprantis and O. Burkinshaw, {\it Dunford-Pettis operators on Banach Lattices}, Transactions of the American Mathemathcal Society, {\bf 274} (1982), 227--238.

\vspace*{-0.5em}
\bibitem{bbjp} G. Botelho, H.-A. Braunss, H. Junek and D. Pellegrino, {\it Holomorphy types and ideals of multilinear mappings},  Studia Math. {\bf 177} (2006), 43--65.

\vspace*{-0.5em}
\bibitem{bp12} G. Botelho and D. Pellegrino, D, {\it Two new properties of ideals of polynomials and applications}, Indag. Math. (NS) {\bf 16} (2005), 157--169.

\vspace*{-0.5em}
\bibitem{bpr2} G. Botelho, D. Pellegrino and P. Rueda, {\it On composition ideals of multilinear operators and homogeneous polynomials},  Publ. Res. Inst. Math. Sci. {\bf 43} (2007), 1139--1155.
%

\vspace*{-0.5em}
\bibitem{leticia} G. Botelho and L. Polac, {\it A polynomial Hutton theorem with applications.}, J. Math. Anal. Appl. {\bf 415} (2014), no.1, 294--301.

\vspace*{-0.5em}
\bibitem{ewerton} G. Botelho and E. R. Torres, {\it Hyper-ideals of multilinear operators}, Linear Algebra Appl. {\bf 482} (2015), 1--20.

\vspace*{-0.5em}
\bibitem{ewerton2} G. Botelho and E. R. Torres, {\it Techniques to generate hyper-ideals of multilinear operators}, Linear Multilinear Algebra {\bf 65} (2017), 1232--1246.

\vspace*{-0.5em}
\bibitem{ewerton3} G. Botelho and E. R. Torres, {\it Two-sided polynomial ideals on Banach spaces}, J. Math. Anal. Appl. {\bf 462} (2018), 900--914.


\vspace*{-0.5em}
\bibitem{B80} J. Bourgain, {\it Dunford-Pettis operators on L1 and the Radon-Nikodym property}, Israel Journal of Mathematics. {\bf 37} (1980), 34--47. 

\vspace*{-0.5em}
\bibitem{cp19} E. \c{C}ali\c{s}kan and D. Pellegrino, {\it On the multilinear generalizations of the concept of absolutely summing operators}, Rocky Mt. J. Math. {\bf 37} (2007), 1137--1154.

\vspace*{-0,5em}
\bibitem{CDM09} D. Carando, V. Dimant and S. Muro, {\it Coherent sequences of polynomial ideals on Banach spaces}. Mathematische Nachrichten. {\bf 282} (2009), 1111--1133.

\vspace*{-0,5em}
\bibitem{cdm22} D. Carando, V. Dimant and S. Muro, {\it Holomorphic functions and polynomial ideals on Banach spaces}, Collect. Math. {\bf 63} (2012), 71--91.

\vspace*{-0,5em}
\bibitem{cdm23} D. Carando, V. Dimant and S. Muro, {\it Every Banach ideal of polynomials is compatible with an operator ideal}, Monatsh. Math. {\bf 165} (2012), 1--14.

\vspace*{-0.5em}
\bibitem{dimant} V. Dimant, {\it Strongly p-summing multilinear operators}, J. Math. Anal. Appl. {\bf 278} (2003), 182--193.

\vspace*{-0.5em}
\bibitem{dineen} S. Dineen, {\it Complex Analysis on Infinite Dimensional Spaces}, Springer, (1999).

\vspace*{-0.5em}
\bibitem{diestel} J. Diestel, H. Jarchow and A. Tonge, {\it Absolutely Summing Operators}, Cambridge University Press, New York, (1995).

\vspace*{-0.5em}
\bibitem{matos47} M. C. Matos, {\it Fully absolutely summing mappings and Hilbert Schmidt operators}, Collect. Math. {\bf 54} (2003), 111--136.

\vspace*{-0.5em}
\bibitem{matos48} M. C. Matos, {\it Nonlinear absolutely summing mappings}, Math. Nachr. {\bf 258} (2003), 71--89.

\vspace*{-0.5em}
\bibitem{mujicatrans} J. Mujica, {\it Linearization of bounded holomorphic mappings on Banach spaces}, Trans. Amer. Math. Soc. {\bf 324}  (1991),  867--887.

\vspace*{-0.5em}
\bibitem{mujica} J. Mujica, {\it Complex Analysis in Banach Spaces}, Dover Publications, Mineola, New York, (2010).

\vspace*{-0.5em}
\bibitem{joilson} D. Pellegrino and J. Ribeiro, {\it On multi-ideals and polynomial ideals of Banach spaces: a new approach to coherence and compatibility},  Monatsh. Math. {\bf 173} (2014), no. 3, 379--415.

\vspace*{-0.5em}
\bibitem{pelczynski}A. Pe{\l}czy\'nski, {\it On weakly compact polynomial operators on B-spaces with Dunford-Pettis Property},  Bull. Acad. Polon. Sci. S\'{e}r. Sci. Math. Astronom. Phys. {\bf 11} (1963), 371--377.

\vspace*{-0.5em}
\bibitem{ps54} D. Pellegrino and J. Santos, {\it Absolutely summing operators: a panorama}, Quaest. Math. {\bf 34} (2011), 447--478.

\vspace*{-0.5em}
\bibitem{pg60} D. P\'{e}rez-Garc\'{i}a, {\it Comparing different classes of absolutely summing multilinear operators}, Arch. Math. {\bf 85} (2005), 258--267.

\vspace*{-0.5em}
\bibitem{pietsch} A. Pietsch, {\it Ideals of multilinear functionals}, Proceedings of the Second International Conference on Operator Algebras, Ideals and Their Applications in Theoretical Physics, Leipzig Teubner Texte Math. {\bf 62} (1983), 185--199.

\vspace*{-0.5em}
\bibitem{ryanthesis} R. A. Ryan, {\it Applications of topological tensor products to infinite dimensional holomorphy}, Thesis, Trinity College Dublin, (1980).

%
%
%
%
%
%
%
%
\vspace{2em}

\end{thebibliography}
\end{document}